\documentclass[12pt]{article}

\usepackage{amsmath,amsthm,amssymb}
\usepackage[top=30truemm,bottom=30truemm,left=25truemm,right=25truemm]{geometry}
\usepackage[dvips]{graphicx,color,psfrag}
\usepackage{bm}

\theoremstyle{plain}
\newtheorem{definition}{Definition}[section]
\newtheorem{thm}[definition]{Theorem}

\newtheorem{lem}[definition]{Lemma}
\newtheorem{cor}[definition]{Corollary}

\newtheorem{rem}[definition]{Remark}

\title{Integer group determinants for abelian groups \\ of order $16$}
\author{Yuka Yamaguchi and Naoya Yamaguchi}
\date{\today}

\begin{document}

\maketitle

\begin{abstract}
For any positive integer $n$, 
let ${\rm C}_{n}$ be the cyclic group of order $n$. 
We determine all possible values of the integer group determinant of ${\rm C}_{4} \times {\rm C}_{2}^{2}$, 
which is the only unsolved abelian group of order $16$. 
\end{abstract}

\section{Introduction}
A circulant determinant is the determinant of a square matrix in which each row is obtained by a cyclic shift of the previous row one step to the right. 
At the meeting of the American Mathematical Society in Hayward, California, in April 1977, 
Olga Taussky-Todd \cite{Olga} suggested a problem that is to determine the all possible integers of an $n \times n$ circulant determinant when the all entries are integers (see e.g., \cite{ https://doi.org/10.48550/arxiv.2205.12439, MR550657}). 
The solution for the case $n = 2$ is well known. 
In the cases of $n = p$ and $2 p$, 
where $p$ is an odd prime, 
the problem was solved \cite{MR624127, MR550657}. 
Also, the problem was solved for the cases $n = 9$ \cite[Theorem~4]{MR601702}, $n = 4$ and $8$ \cite[Theorem~1.1]{MR2914452}, 
$n = 12$ \cite[Theorem~5.3]{MR4056860}, 
$n = 15$ \cite[Theorem~1.3]{MR4363104} and $n = 16$ \cite{Yamaguchi}. 
For a finite group $G$, 
let $x_{g}$ be a variable for each $g \in G$. 
The group determinant of $G$ is defined as $\det{\left( x_{g h^{- 1}} \right)}_{g, h \in G}$. 
Let ${\rm C}_{n}$ be the cyclic group of order $n$. 
Note that the group determinant of ${\rm C}_{n}$ becomes an $n \times n$ circulant determinant. 
The group determinant of $G$ is called an integer group determinant of $G$ when the variables $x_{g}$ are all integers. 
Let $S(G)$ denote the set of all possible values of the integer group determinant of $G$: 
$$
S(G) := \left\{ \det{\left( x_{g h^{- 1}} \right)}_{g, h \in G} \mid x_{g} \in \mathbb{Z} \right\}. 
$$
The problem suggested by Olga Taussky-Todd is extended to the problem that is to determine $S(G)$ for any finite group $G$. 
For some groups, the problem was solved in \cite{MR3879399, MR2914452, MR624127, MR601702, MR4363104, MR3998922, MR4056860, https://doi.org/10.48550/arxiv.2203.14420, Yamaguchi, https://doi.org/10.48550/arxiv.2209.12446, https://doi.org/10.48550/arxiv.2211.01597}. 
As a result, for every group $G$ of order at most $15$, $S(G)$ was determined. 
Also, ${\rm C}_{4} \times {\rm C}_{2}^{2}$ is left as the only unsolved abelian group of order $16$. 

In this paper,  we determine $S \left( {\rm C}_{4} \times {\rm C}_{2}^{2} \right)$. 
For any $r \in \mathbb{Z}$, let 
\begin{align*}
P_{r} &:= \left\{ p \mid p \equiv r \: \: ({\rm mod} \: 8) \: \: \text{is a prime number} \right\}, \\ 
P' &:= \left\{ p \mid p = a^{2} + b^{2} \equiv 1 \: \: ({\rm mod} \: 8) \: \: \text{is a prime number satisfying} \: \: a + b \equiv \pm 3 \: \: ({\rm mod} \: 8) \right\}, \\ 
A &:= \left\{ (8 k - 3) (8 l + 3) \mid k, l \in \mathbb{Z} \right\} \subsetneq \left\{ 8 m - 1 \mid m \in \mathbb{Z} \right\}, \\ 
B &:= \left\{ p (8 m - 1) \mid p \in P', \: m \in \mathbb{Z} \right\} \subsetneq \left\{ 8 m - 1 \mid m \in \mathbb{Z} \right\}. 
\end{align*}

\begin{thm}\label{thm:1.1}
We have 
\begin{align*}
S({\rm C}_{4} \times {\rm C}_{2}^{2}) &= \left\{ 16 m + 1, \: 2^{16} (4 m + 1), \: 2^{16} (8 m + 3), \: 2^{17} p (2 m + 1), \: 2^{18} m \mid m \in \mathbb{Z}, \: p \in P_{5} \right\} \\ 
&\qquad \cup \left\{ 2^{16} m \mid m \in A \cup B \right\}. 
\end{align*}
\end{thm}

Let ${\rm D}_{n}$ denote the dihedral group of order $n$ and let 
\begin{align*}
C &:= \left\{ (8 k - 3) (8 l - 3) (8 m - 3) (8 n - 3) \mid k \in \mathbb{Z}, \: 8 l - 3, 8 m - 3, 8 n - 3 \in P_{5}, \right. \\ 
&\qquad \left. k + l \not\equiv m + n \: \: ({\rm mod \: 2}) \right\} \subsetneq \left\{ 16 m - 7 \mid m \in \mathbb{Z} \right\}, \\ 
D &:= \left\{ (8 k - 3) (8 l - 3) \mid k, l \in \mathbb{Z}, \: k \equiv l \: \: ({\rm mod} \: 2) \right\} \subsetneq \left\{ 16 m - 7 \mid m \in \mathbb{Z} \right\}. 
\end{align*}
Remark that $C \subset D$ holds. 
In \cite{https://doi.org/10.48550/arxiv.2209.12446}, \cite{https://doi.org/10.48550/arxiv.2211.01597}, \cite[Theorem~1.5]{https://doi.org/10.48550/arxiv.2203.14420}, \cite[Theorem~5.3]{MR3879399} and \cite{Yamaguchi}, the following are obtained respectively: 
\begin{align*}
S({\rm C}_{2}^{4}) &= \left\{ 16 m + 1, \: 2^{16} (4 m + 1), \: 2^{24} (4 m + 1), \: 2^{24} (8 m + 3), \: 2^{24} m', \: 2^{26} m \mid \right. \\ 
&\qquad \left. m \in \mathbb{Z}, \: m' \in A \right\}, \\ 
S({\rm C}_{4}^{2}) &= \left\{ 16 m + 1, \: m', \: 2^{15} p (2 m + 1), \: 2^{16} m \mid m \in \mathbb{Z}, \: m' \in C, \: p \in P_{5} \right\}, \\ 
S({\rm C}_{8} \times {\rm C}_{2}) &= \left\{ 16 m + 1, \: m', \: 2^{10} (2 m + 1), \: 2^{11} p (2 m + 1), \: 2^{11} q^{2} (2 m + 1), \: 2^{12} m \mid \right. \\ 
&\qquad \left. m \in \mathbb{Z}, \: m' \in D, \: p \in P' \cup P_{5}, \: q \in P_{3}  \right\}, \\ 
S({\rm D}_{16}) &= \left\{ 4 m + 1, \: 2^{10} m \mid m \in \mathbb{Z} \right\}, \\ 
S({\rm C}_{16}) &= \left\{ 2 m + 1, \: 2^{6} p (2 m + 1), \: 2^{6} q^{2} (2 m + 1), \: 2^{7} m \mid m \in \mathbb{Z}, \: p \in P' \cup P_{5}, \: q \in P_{3} \right\}. 
\end{align*}
Pinner and Smyth \cite[p.427]{MR4056860} noted the following inclusion relations for every groups of order $8$: 
$S({\rm C}_{2}^{3}) \subsetneq S({\rm C}_{4} \times {\rm C}_{2}) \subsetneq S({\rm Q}_{8}) \subsetneq S({\rm D}_{8}) \subsetneq S({\rm C}_{8})$, 
where ${\rm Q}_{8}$ denotes the generalized quaternion group of order $8$. 
 From the above results, 
 we have 
 $$
 S({\rm C}_{2}^{4}) \subsetneq S({\rm C}_{4} \times {\rm C}_{2}^{2}) \subsetneq S({\rm C}_{4}^{2}) \subsetneq S({\rm C}_{8} \times {\rm C}_{2}) \subsetneq S({\rm D}_{16}) \subsetneq S({\rm C}_{16}). 
 $$


\section{Preliminaries}
For any $\overline{r} \in {\rm C}_{n}$ with $r \in \{ 0, 1, \ldots, n - 1 \}$, 
we denote the variable $x_{\overline{r}}$ by $x_{r}$, 
and let $D_{n}(x_{0}, x_{1}, \ldots, x_{n - 1}) := \det{\left( x_{g h^{- 1}} \right)}_{g, h \in {\rm C}_{n}}$. 
For any $(\overline{r}, \overline{s}) \in {\rm C}_{4} \times {\rm C}_{2}$ with $r \in \{ 0, 1, 2, 3 \}$ and $s \in \{ 0, 1 \}$, 
we denote the variable $y_{(\overline{r}, \overline{s})}$ by $y_{j}$, where $j := r + 4 s$, 
and let $D_{4 \times 2}(y_{0}, y_{1}, \ldots, y_{7}) := \det{\left( y_{g h^{- 1}} \right)}_{g, h \in {\rm C}_{4} \times {\rm C}_{2}}$. 
For any $(\overline{r}, \overline{s}, \overline{t}) \in {\rm C}_{4} \times {\rm C}_{2}^{2}$ with $r \in \{ 0, 1, 2, 3 \}$ and $s, t \in \{ 0, 1 \}$, 
we denote the variable $z_{(\overline{r}, \overline{s}, \overline{t})}$ by $z_{j}$, where $j := r + 4 s + 8 t$, 
and let $D_{4 \times 2 \times 2}(z_{0}, z_{1}, \ldots, z_{15}) := \det{\left( z_{g h^{- 1}} \right)}_{g, h \in {\rm C}_{4} \times {\rm C}_{2}^{2}}$. 
From the $G = {\rm C}_{4}$ and $H = \{ \overline{0}, \overline{2} \}$ case of \cite[Theorem~1.1]{https://doi.org/10.48550/arxiv.2203.14420}, 
we have the following corollary. 

\begin{cor}\label{cor:2.1}
We have 
\begin{align*}
D_{4}(x_{0}, x_{1}, x_{2}, x_{3}) 
&= D_{2}(x_{0} + x_{2}, x_{1} + x_{3}) D_{2}\left( x_{0} - x_{2}, \sqrt{- 1} (x_{1} - x_{3}) \right) \\ 
&= \left\{ (x_{0} + x_{2})^{2}- (x_{1} + x_{3})^{2} \right\} \left\{ (x_{0} - x_{2})^{2} + (x_{1} - x_{3})^{2} \right\}. 
\end{align*}
\end{cor}

\begin{rem}\label{rem:2.2}
From Corollary~$\ref{cor:2.1}$, we have $D_{4}(x_{0}, x_{1}, x_{2}, x_{3}) = - D_{4}(x_{1}, x_{2}, x_{3}, x_{0})$. 
\end{rem}

From \cite[Theorem~1.1]{https://doi.org/10.48550/arxiv.2202.06952}, 
we have the following corollary. 

\begin{cor}\label{cor:2.3}
The following hold: 
\begin{enumerate}
\item[$(1)$] Let $H = {\rm C}_{4}$ and $K = {\rm C}_{2}$. Then we have 
$$
D_{4 \times 2}(y_{0}, y_{1}, \ldots, y_{7}) = D_{4}(y_{0} + y_{4}, y_{1} + y_{5}, y_{2} + y_{6}, y_{3} + y_{7}) D_{4}(y_{0} - y_{4}, y_{1} - y_{5}, y_{2} - y_{6}, y_{3} - y_{7}); 
$$
\item[$(2)$] Let $H = {\rm C}_{4} \times {\rm C}_{2}$ and $K = {\rm C}_{2}$. Then we have 
$$
D_{4 \times 2 \times 2}(z_{0}, z_{1}, \ldots, z_{15}) = D_{4 \times 2}(z_{0} + z_{8}, z_{1} + z_{9}, \ldots, z_{7} + z_{15}) D_{4 \times 2}(z_{0} - z_{8}, z_{1} - z_{9}, \ldots, z_{7} - z_{15}). 
$$
\end{enumerate}
\end{cor}

Throughout this paper, 
we assume that $a_{0}, a_{1}, \ldots, a_{15} \in \mathbb{Z}$, 
and for any $0 \leq i \leq 3$, let 
\begin{align*}
b_{i} &:= (a_{i} + a_{i + 8}) + (a_{i + 4} + a_{i + 12}), & 
c_{i} &:= (a_{i} + a_{i + 8}) - (a_{i + 4} + a_{i + 12}), \\ 
d_{i} &:= (a_{i} - a_{i + 8}) + (a_{i + 4} - a_{i + 12}), & 
e_{i} &:= (a_{i} - a_{i + 8}) - (a_{i + 4} - a_{i + 12}). 
\end{align*}
Also, let $\bm{a} := (a_{0}, a_{1}, \ldots, a_{15})$ and let 
\begin{align*}
\bm{b} := (b_{0}, b_{1}, b_{2}, b_{3}), \quad 
\bm{c} := (c_{0}, c_{1}, c_{2}, c_{3}), \quad 
\bm{d} := (d_{0}, d_{1}, d_{2}, d_{3}), \quad 
\bm{e} := (e_{0}, e_{1}, e_{2}, e_{3}). 
\end{align*}
Then, from Corollary~$\ref{cor:2.3}$, we have 
$$
D_{4 \times 2 \times 2}(\bm{a}) = D_{4}(\bm{b}) D_{4}(\bm{c}) D_{4}(\bm{d}) D_{4}(\bm{e}). 
$$

\begin{rem}\label{rem:2.4}
For any $0 \leq i \leq 3$, the following hold: 
\begin{enumerate}
\item[$(1)$] $b_{i} \equiv c_{i} \equiv d_{i} \equiv e_{i} \pmod{2}$; 
\item[$(2)$] $b_{i} + c_{i} + d_{i} + e_{i} \equiv 0 \pmod{4}$. 
\end{enumerate}
\end{rem}

\begin{lem}\label{lem:2.5}
We have $D_{4 \times 2 \times 2}(\bm{a}) \equiv D_{4}(\bm{b}) \equiv D_{4}(\bm{c}) \equiv D_{4}(\bm{d}) \equiv D_{4}(\bm{e}) \pmod{2}$. 
\end{lem}
\begin{proof}
From Corollary~$\ref{cor:2.1}$, 
for any $x_{0}, x_{1}, x_{2}, x_{3} \in \mathbb{Z}$, 
\begin{align*}
D_{4}(x_{0}, x_{1}, x_{2}, x_{3}) \equiv x_{0} + x_{1} + x_{2} + x_{3} \pmod{2}. 
\end{align*}
Therefore, from Remark~$\ref{rem:2.4}$~(1), we have $D_{4}(\bm{b}) \equiv D_{4}(\bm{c}) \equiv D_{4}(\bm{d}) \equiv D_{4}(\bm{e}) \pmod{2}$. 
\end{proof}

\section{Impossible odd numbers}\label{section3}
In this section, 
we consider impossible odd numbers. 
Let $\mathbb{Z}_{\rm odd}$ be the set of all odd numbers. 

\begin{lem}\label{lem:3.1}
We have $S({\rm C}_{4} \times {\rm C}_{2}^{2}) \cap \mathbb{Z}_{\rm odd} \subset \{ 16 m + 1 \mid m \in \mathbb{Z} \}$. 
\end{lem}

To prove Lemma~$\ref{lem:3.1}$, we use the following lemma. 

\begin{lem}[{\cite[Lemmas~4.6 and 4.7]{https://doi.org/10.48550/arxiv.2203.14420}}]\label{lem:3.2}
For any $k, l, m, n \in \mathbb{Z}$, the following hold: 
\begin{enumerate}
\item[$(1)$] $D_{4}(2 k + 1, 2 l, 2 m, 2 n) \equiv 8 m + 1 \pmod{16}$; 
\item[$(2)$] $D_{4}(2 k, 2 l + 1, 2 m + 1, 2 n + 1) \equiv 8 (k + l + n) - 3 \pmod{16}$. 
\end{enumerate}
\end{lem}

\begin{proof}[Proof of Lemma~$\ref{lem:3.1}$]
Let $D_{4 \times 2 \times 2}(\bm{a}) = D_{4}(\bm{b}) D_{4}(\bm{c}) D_{4}(\bm{d}) D_{4}(\bm{e}) \in \mathbb{Z}_{\rm odd}$. 
Then $D_{4}(\bm{b}) \in \mathbb{Z}_{\rm odd}$. 
From this and Corollary~$\ref{cor:2.1}$, we have $b_{0} + b_{2} \not\equiv b_{1} + b_{3} \pmod{2}$. 
Therefore, either one of the following cases holds: 
\begin{enumerate}
\item[(i)] exactly three of $b_{0}, b_{1}, b_{2}, b_{3}$ are even; 
\item[(ii)] exactly one of $b_{0}, b_{1}, b_{2}, b_{3}$ is even. 
\end{enumerate}
First, we consider the case~(i). 
From Remarks~$\ref{rem:2.2}$ and $\ref{rem:2.4}$~(1), 
we may assume without loss of generality that $\bm{b} \equiv \bm{c} \equiv \bm{d} \equiv \bm{e} \equiv (1, 0, 0, 0) \pmod{2}$. 
From Remark~$\ref{rem:2.4}$, there exist $m_{i} \in \mathbb{Z}$ satisfying $b_{2} = 2 m_{0}$, $c_{2} = 2 m_{1}$, $d_{2} = 2 m_{2}$, $e_{2} = 2 m_{3}$ 
and $\sum_{i = 0}^{3} m_{i} \equiv 0 \pmod{2}$. Therefore, from Lemma~$\ref{lem:3.2}$~(1), 
\begin{align*}
D_{4 \times 2 \times 2}(\bm{a}) \equiv \prod_{i = 0}^{3} (8 m_{i} + 1) \equiv 8 \sum_{i = 0}^{3} m_{i} + 1 \equiv 1 \pmod{16}. 
\end{align*}
Next, we consider the case~(ii). 
From Remarks~$\ref{rem:2.2}$ and $\ref{rem:2.4}$~(1), 
we may assume without loss of generality that $\bm{b} \equiv \bm{c} \equiv \bm{d} \equiv \bm{e} \equiv (0, 1, 1, 1) \pmod{2}$. 
From Remark~$\ref{rem:2.4}$, 
there exist $k_{i}, l_{i}, n_{i} \in \mathbb{Z}$ satisfying 
\begin{align*}
(b_{0}, b_{1}, b_{2}) &= (2 k_{0}, 2 l_{0} + 1, 2 n_{0} + 1), &
(c_{0}, c_{1}, c_{3}) &= (2 k_{1}, 2 l_{1} + 1, 2 n_{1} + 1), \\ 
(d_{0}, d_{1}, d_{3}) &= (2 k_{2}, 2 l_{2} + 1, 2 n_{2} + 1), &
(e_{0}, e_{1}, e_{3}) &= (2 k_{3}, 2 l_{3} + 1, 2 n_{3} + 1) 
\end{align*}
and $\sum_{i = 0}^{3} k_{i} \equiv \sum_{i = 0}^{3} l_{i} \equiv \sum_{i = 0}^{3} n_{i} \equiv 0 \pmod{2}$. 
Therefore, from Lemma~$\ref{lem:3.2}$~(2), we have 
\begin{align*}
D_{4 \times 2 \times 2}(\bm{a}) \equiv \prod_{i = 0}^{3} \left\{ 8 (k_{i} + l_{i} + n_{i}) - 3 \right\} \equiv 8 \sum_{i = 0}^{3} (k_{i} + l_{i} + n_{i}) + 1 \equiv 1 \pmod{16}. 
\end{align*}
\end{proof}

\section{Impossible even numbers}\label{Section4}

In this section, 
we consider impossible even numbers. 
Let 
\begin{align*}
P &:= \left\{ p \mid p \equiv - 3 \: \: ({\rm mod} \: 8) \: \: \text{is a prime number} \right\}, \\ 
P' &:= \left\{ p \mid p = a^{2} + b^{2} \equiv 1 \: \: ({\rm mod} \: 8) \: \: \text{is a prime number satisfying} \: \: a + b \equiv \pm 3 \: \: ({\rm mod} \: 8) \right\}, \\ 
A &:= \left\{ (8 k - 3) (8 l + 3) \mid k, l \in \mathbb{Z} \right\}, \\ 
B &:= \left\{ p (8 m - 1) \mid p \in P', \: m \in \mathbb{Z} \right\}. 
\end{align*}

\begin{lem}\label{lem:4.1}
We have $S({\rm C}_{4} \times {\rm C}_{2}^{2}) \cap 2 \mathbb{Z} \subset 2^{16} \mathbb{Z}$. 
\end{lem}

\begin{lem}\label{lem:4.2}
We have 
$$
S({\rm C}_{4} \times {\rm C}_{2}^{2}) \cap 2^{16} \mathbb{Z}_{\rm odd} \subset \left\{ 2^{16} (4 m + 1), \: 2^{16} (8 m + 3), \: 2^{16} m' \mid m \in \mathbb{Z}, \: m' \in A \cup B \right\}. 
$$
\end{lem}

\begin{lem}\label{lem:4.3}
We have $S({\rm C}_{4} \times {\rm C}_{2}^{2}) \cap 2^{17} \mathbb{Z}_{\rm odd} \subset \left\{ 2^{17} p (2 m + 1) \mid p \in P, \: m \in \mathbb{Z} \right\}$. 
\end{lem}

Lemma~$\ref{lem:4.1}$ is immediately obtained from Lemma~$\ref{lem:2.5}$ and Kaiblinger's \cite[Theorem~1.1]{MR2914452} result $S({\rm C}_{4}) = \mathbb{Z}_{\rm odd} \cup 2^{4} \mathbb{Z}$. 
To prove Lemma~$\ref{lem:4.2}$, 
we use the following six lemmas. 

\begin{lem}[{\cite[Lemma~$3.2$]{https://doi.org/10.48550/arxiv.2211.01597}}]\label{lem:4.4}
For any $k, l, m, n \in \mathbb{Z}$, the following hold: 
\begin{enumerate}
\item[$(1)$] $D_{4}(2 k, 2 l, 2 m, 2 n) \in \begin{cases} 2^{4} \mathbb{Z}_{\rm odd}, & k + m \not\equiv l + n \pmod{2}, \\ 2^{8} \mathbb{Z}, & k + m \equiv l + n \pmod{2}; \end{cases}$
\item[$(2)$] $D_{4}(2 k + 1, 2 l + 1, 2 m + 1, 2 n + 1) \in \begin{cases} 2^{4} \mathbb{Z}_{\rm odd}, & k + m \not\equiv l + n \pmod{2}, \\ 2^{7} \mathbb{Z}_{\rm odd}, & (k + m) (l + n) \equiv - 1 \pmod{4}, \\ 2^{9} \mathbb{Z}, & \text{otherwise}; \end{cases}$
\item[$(3)$] $D_{4}(2 k, 2 l + 1, 2 m, 2 n + 1)$ \\ 
\quad $\in \begin{cases} 2^{5} \mathbb{Z}_{\rm odd}, & k - m \equiv l - n \equiv 1 \pmod{2}, \\ 2^{6} \mathbb{Z}_{\rm odd}, & k \equiv m \pmod{2} \: \: \text{and} \: \: (2 k + 2 l + 1) (2 m + 2 n + 1) \equiv \pm 3 \pmod{8}, \\ 2^{7} \mathbb{Z}, & \text{otherwise}; \end{cases}$
\item[$(4)$] $D_{4}(2 k, 2 l, 2 m + 1, 2 n + 1) \in \begin{cases} 2^{4} \mathbb{Z}_{\rm odd}, & (2 k + 2 m + 1) (2 l + 2 n + 1) \equiv \pm 3 \pmod{8}, \\ 2^{5} \mathbb{Z}, & (2 k + 2 m + 1) (2 l + 2 n + 1) \equiv \pm 1 \pmod{8}.  \end{cases}$
\end{enumerate}
\end{lem}

\begin{lem}\label{lem:4.5}
Suppose that $(x_{0}, x_{1}, x_{2}, x_{3}) \equiv (0, 0, 1, 1) \pmod{2}$ and $(x_{0} + x_{2}) (x_{1} + x_{3}) \equiv \pm 3 \pmod{8}$ hold. 
Then the following hold: 
\begin{enumerate}
\item[$(1)$] If $x_{0} \equiv x_{1} \pmod{4}$, then $(x_{0} - x_{2})^{2} + (x_{1} - x_{3})^{2} \in \left\{ 2 (8 k - 3) \mid k \in \mathbb{Z} \right\}$; 
\item[$(2)$] If $x_{0} \not\equiv x_{1} \pmod{4}$, then $(x_{0} - x_{2})^{2} + (x_{1} - x_{3})^{2} \in \left\{ 2 (8 k + 1) \mid k \in \mathbb{Z} \right\}$. 
\end{enumerate}
\end{lem}
\begin{proof}
We prove (1). 
If $x_{0} \equiv x_{1} \pmod{4}$, then 
\begin{align*}
(x_{0} - x_{2}) (x_{1} - x_{3}) 
&= (x_{0} + x_{2}) (x_{1} + x_{3}) - 2 x_{0} x_{3} - 2 x_{2} x_{1} \\ 
&\equiv (x_{0} + x_{2}) (x_{1} + x_{3}) - 2 x_{0} (x_{3} + x_{2}) \\ 
&\equiv \pm 3 \pmod{8}. 
\end{align*}
Thus, $(x_{0} - x_{2})^{2} + (x_{1} - x_{3})^{2} \equiv - 6 \pmod{16}$ holds. 
We prove (2). 
If $x_{0} \not\equiv x_{1} \pmod{4}$, then 
\begin{align*}
(x_{0} - x_{2}) (x_{1} - x_{3}) 
&= (x_{0} + x_{2}) (x_{1} + x_{3}) - 2 x_{0} x_{3} - 2 x_{2} x_{1} \\ 
&\equiv (x_{0} + x_{2}) (x_{1} + x_{3}) - 2 x_{0} x_{3} - 2 x_{2} (x_{0} + 2) \\ 
&\equiv (x_{0} + x_{2}) (x_{1} + x_{3}) - 2 x_{0} (x_{3} + x_{2}) - 4 x_{2} \\ 
&\equiv \pm 1 \pmod{8}. 
\end{align*}
Thus, $(x_{0} - x_{2})^{2} + (x_{1} - x_{3})^{2} \equiv 2 \pmod{16}$ holds. 
\end{proof}

\begin{lem}\label{lem:4.6}
Suppose that $(x_{0} + x_{2})^{2} - (x_{1} + x_{3})^{2}$ has no prime factor of the form $8 k \pm 3$. 
Then the following hold: 
\begin{enumerate}
\item[$(1)$] If $x_{0} + x_{2} \equiv \pm 3, \: \: x_{1} + x_{3} \equiv \pm 1 \pmod{8}$, then $(x_{0} + x_{2})^{2} - (x_{1} + x_{3})^{2} \in \left\{ 8 (8 k + 1) \mid k \in \mathbb{Z} \right\}$; 
\item[$(2)$] If $x_{0} + x_{2} \equiv \pm 1, \: \: x_{1} + x_{3} \equiv \pm 3 \pmod{8}$, then $(x_{0} + x_{2})^{2} - (x_{1} + x_{3})^{2} \in \left\{ 8 (8 k - 1) \mid k \in \mathbb{Z} \right\}$. 
\end{enumerate}
\end{lem}
\begin{proof}
We prove (1). 
First, we consider the case of $(x_{0} + x_{2}, x_{1} + x_{3}) \equiv (3, 1) \pmod{8}$. 
Then $(x_{0} + x_{2}, x_{1} + x_{3}) \equiv (3, 1), (3, - 7), (- 5, 1) \: \text{or} \: (- 5, - 7) \pmod{16}$. 
From 
\begin{align*}
(16 l + 3)^{2} - (16 m + 1)^{2} 
&= (16 l + 3 + 16 m + 1) (16 l + 3 - 16 m - 1) \\ 
&= 8 (4 l + 4 m + 1) (8 l - 8 m + 1), \\ 
(16 l + 3)^{2} - (16 m - 7)^{2} 
&= (16 l + 3 + 16 m - 7) (16 l + 3 - 16 m + 7) \\ 
&= 8 (4 l + 4 m - 1) (8 l - 8 m + 5), \\ 
(16 l - 5)^{2} - (16 m + 1)^{2} 
&= (16 l - 5 + 16 m + 1) (16 l - 5 - 16 m - 1) \\ 
&= 8 (4 l + 4 m - 1) (8 l - 8 m - 3), \\ 
(16 l - 5)^{2} - (16 m - 7)^{2} 
&= (16 l - 5 + 16 m - 7) (16 l - 5 - 16 m + 7) \\ 
&= 8 (4 l + 4 m - 3) (8 l - 8 m + 1), 
\end{align*}
we find that if $(x_{0} + x_{2}, x_{1} + x_{3}) \equiv (3, - 7) \: \text{or} \: (- 5, 1) \pmod{16}$, 
then $(x_{0} + x_{2})^{2} - (x_{1} + x_{3})^{2}$ has at least one prime factor of the form $8 k \pm 3$. 
Also, if $(x_{0} + x_{2}, x_{1} + x_{3}) \equiv (3, 1) \: \text{or} \: (- 5, - 7) \pmod{16}$, 
then $(x_{0} + x_{2})^{2} - (x_{1} + x_{3})^{2}$ is of the form $8 (8 k + 1)$ or has at least one prime factor of the form $8 k \pm 3$. 
In the same way, we can prove for the cases $(x_{0} + x_{2}, x_{1} + x_{3}) \equiv (3, - 1), (- 3, 1)$ and $(- 3, - 1) \pmod{8}$. 
Replacing $(x_{0}, x_{1}, x_{2}, x_{3})$ with $(x_{1}, x_{2}, x_{3}, x_{0})$ in (1), we obtain (2). 
\end{proof}

\begin{lem}\label{lem:4.7}
Suppose that $(x_{0} - x_{2})^{2} + (x_{1} - x_{3})^{2}$ has no prime factor of the form $8 k \pm 3$ and $x_{0} - x_{2} \equiv \pm 3, \: x_{1} - x_{3} \equiv \pm 3 \pmod{8}$ hold. 
Then $(x_{0} - x_{2})^{2} + (x_{1} - x_{3})^{2}$ has at least one prime factor of the form $p = a^{2} + b^{2} \equiv 1 \pmod{8}$ with $a + b \equiv \pm 3 \pmod{8}$. 
\end{lem}
\begin{proof}
From the assumption, 
there exist primes $p_{i} \equiv 1, \: q_{i} \equiv - 1 \pmod{8}$ and integers $k_{i}, l_{i} \geq 0$ satisfying $
(x_{0} - x_{2})^{2} + (x_{1} - x_{3})^{2} = 2 p_{1}^{k_{1}} \cdots p_{r}^{k_{r}} q_{1}^{2 l_{1}} \cdots q_{s}^{2 l_{s}}$. 
We prove by contradiction. 
If $p_{i} = a_{i}^{2} + b_{i}^{2}$ with $a_{i} + b_{i} \equiv \pm 1 \pmod{8}$ for any $1 \leq i \leq r$, 
then $x_{0} - x_{2}, x_{1} - x_{3} \in \left\{ 8 k \pm 1 \mid k \in \mathbb{Z} \right\}$ hold from \cite[Lemma~4.8]{Yamaguchi}. 
This is a contradiction. 
\end{proof}

\begin{lem}\label{lem:4.8}
Suppose that $b_{0} + b_{2} \equiv  b_{1} + b_{3} \equiv 0 \pmod{2}$ and $D_{4 \times 2 \times 2}(\bm{a}) \in 2^{16} \mathbb{Z}_{{\rm odd}}$. 
Then we have $D_{4 \times 2 \times 2}(\bm{a}) \in \left\{ 2^{16} (4 m + 1) \mid m \in \mathbb{Z} \right\}$. 
\end{lem}
\begin{proof}
From Remarks~$\ref{rem:2.2}$ and $\ref{rem:2.4}$~(1) and Lemma~$\ref{lem:4.4}$, 
either one of the following cases holds: 
\begin{enumerate}
\item[(i)] $b_{0} \equiv b_{1} \equiv b_{2} \equiv b_{3} \equiv 0 \pmod{2}$ and $D_{4}(\bm{b}), D_{4}(\bm{c}), D_{4}(\bm{d}), D_{4}(\bm{e}) \in 2^{4} \mathbb{Z}_{{\rm odd}}$; 
\item[(ii)] $b_{0} \equiv b_{1} \equiv b_{2} \equiv b_{3} \equiv 1 \pmod{2}$ and $D_{4}(\bm{b}), D_{4}(\bm{c}), D_{4}(\bm{d}), D_{4}(\bm{e}) \in 2^{4} \mathbb{Z}_{{\rm odd}}$. 
\end{enumerate}
First, we consider the case (i). 
From Remark~$\ref{rem:2.4}$, there exist $k_{i}, l_{i}, m_{i}, n_{i} \in \mathbb{Z}$ satisfying 
\begin{align*}
\bm{b} &= (2 k_{0}, 2 l_{0}, 2 m_{0}, 2 n_{0}), &
\bm{c} &= (2 k_{1}, 2 l_{1}, 2 m_{1}, 2 n_{1}), \\ 
\bm{d} &= (2 k_{2}, 2 l_{2}, 2 m _{2}, 2 n_{2}), &
\bm{e} &= (2 k_{3}, 2 l_{3}, 2 m_{3}, 2 n_{3})
\end{align*}
and $\sum_{i = 0}^{3} k_{i} \equiv \sum_{i = 0}^{3} l_{i} \equiv \sum_{i = 0}^{3} m_{i} \equiv \sum_{i = 0}^{3} n_{i} \equiv 0 \pmod{2}$. 
Here, from Lemma~$\ref{lem:4.4}$, $k_{i} + m_{i} \not\equiv l_{i} + n_{i} \pmod{2}$ holds for any $0 \leq i \leq 3$. 
Thus we have 
\begin{align*}
2^{- 4} D_{4}(2 k_{i}, 2 l_{i}, 2 m_{i}, 2 n_{i}) 
&= \left\{ (k_{i} + m_{i})^{2} - (l_{i} + n_{i})^{2} \right\} \left\{ (k_{i} - m_{i})^{2} + (l_{i} - n_{i})^{2} \right\} \\ 
&\equiv (- 1)^{l_{i} + n_{i}} \pmod{4}. 
\end{align*}
Therefore, 
\begin{align*}
2^{- 16} D_{4 \times 2 \times 2}(\bm{a}) 
&= 2^{- 16} \prod_{i = 0}^{3} D_{4}(2 k_{i}, 2 l_{i}, 2 m_{i}, 2 n_{i}) \\ 
&\equiv \prod_{i = 0}^{3} (- 1)^{l_{i} + n_{i}} \\ 
&\equiv (- 1)^{l_{0} + l_{1} + l_{2} + l_{3}} (- 1)^{n_{0} + n_{1} + n_{2} + n_{3}} \\ 
&\equiv 1 \pmod{4}. 
\end{align*}
In the same way, we can prove for the case (ii). 
\end{proof}

\begin{lem}\label{lem:4.9}
Let $b_{0} + b_{2} \equiv b_{1} + b_{3} \equiv 1 \pmod{2}$. 
If $D_{4 \times 2 \times 2}(\bm{a}) \in \left\{ 2^{16} m \mid m \equiv - 1 \pmod{8} \right\}$, 
then $D_{4 \times 2 \times 2}(\bm{a}) \in \left\{ 2^{16} m \mid m \in A \cup B \right\}$. 
\end{lem}
\begin{proof}
Let $D_{4 \times 2 \times 2}(\bm{a}) = D_{4}(\bm{b}) D_{4}(\bm{c}) D_{4}(\bm{d}) D_{4}(\bm{e}) = 2^{16} m$ with $m \equiv - 1 \pmod{8}$. 
From Remarks~$\ref{rem:2.2}$ and $\ref{rem:2.4}$~(1), 
we may assume without loss of generality that $\bm{b} \equiv \bm{c} \equiv \bm{d} \equiv \bm{e} \equiv (0, 0, 1, 1) \pmod{2}$. 
We prove that if $m \not\in A$, then $m \in B$. 
Suppose that $m \not\in A$ and let 
\begin{align*}
Q &:= \left\{ (x_{0}, x_{1}, x_{2}, x_{3}) \in \mathbb{Z}^{4} \mid (x_{0}, x_{1}, x_{2}, x_{3}) \equiv (0, 0, 1, 1) \: \: ({\rm mod} \: 2), \right. \\ 
& \qquad \left. (x_{0} + x_{2}) (x_{1} + x_{3}) \equiv \pm 3 \: \: ({\rm mod} \: 8), \: x_{0} \not\equiv x_{1} \: \: ({\rm mod} \: 4) \right\}, \\ 
Q_{1} &:= \left\{ (x_{0}, x_{1}, x_{2}, x_{3}) \in Q \mid x_{0} + x_{2} \equiv \pm 3, \: x_{1} + x_{3} \equiv \pm 1 \: \: ({\rm mod} \: 8) \right\}, \\ 
Q_{2} &:= \left\{ (x_{0}, x_{1}, x_{2}, x_{3}) \in Q \mid x_{0} + x_{2} \equiv \pm 1, \: x_{1} + x_{3} \equiv \pm 3 \: \: ({\rm mod} \: 8) \right\}, \\ 
Q'_{1} &:= \left\{ (x_{0}, x_{1}, x_{2}, x_{3}) \in Q_{1} \mid x_{0} \equiv 0, \: x_{1} \equiv 2 \: \: ({\rm mod} \: 4) \right\}, \\ 
Q'_{2} &:= \left\{ (x_{0}, x_{1}, x_{2}, x_{3}) \in Q_{2} \mid x_{0} \equiv 2, \: x_{1} \equiv 0 \: \: ({\rm mod} \: 4) \right\}. 
\end{align*}
Since $m \not\in A$, $D_{4}(\bm{b}) D_{4}(\bm{c}) D_{4}(\bm{d}) D_{4}(\bm{e})$ has no prime factor of the form $8 k \pm 3$. 
Thus, from Lemmas~$\ref{lem:4.4}$~(4) and $\ref{lem:4.5}$, 
we have $\bm{b}, \bm{c}, \bm{d}, \bm{e} \in Q$. 
Moreover, from $m \equiv - 1 \pmod{8}$ and Lemmas~$\ref{lem:4.5}$ and $\ref{lem:4.6}$, 
either one of the following cases holds: 
\begin{enumerate}
\item[(i)] One of $\bm{b}, \bm{c}, \bm{d}, \bm{e}$ is an element of $Q_{1}$ and the other three are elements of $Q_{2}$; 
\item[(ii)] One of $\bm{b}, \bm{c}, \bm{d}, \bm{e}$ is an element of $Q_{2}$ and the other three are elements of $Q_{1}$. 
\end{enumerate}
Since $b_{0} + c_{0} + d_{0} + e_{0} \equiv 0 \pmod{4}$ from Remark~$\ref{rem:2.4}$~(2), 
we find that in both cases (i) and (ii), 
at least one of $\bm{b}, \bm{c}, \bm{d}, \bm{e}$ is an element of $Q'_{1} \cup Q'_{2}$. 
On the other hand, for any $\bm{x} = (x_{0}, x_{1}, x_{2}, x_{3}) \in Q'_{1}$, 
we have $- x_{0} + x_{2} \equiv x_{0} + x_{2} \equiv \pm 3, \: - x_{1} + x_{3} \equiv x_{1} + x_{3} - 4 \equiv \pm 3 \: \: ({\rm mod} \: 8)$. 
Thus, it follows from Lemma~$\ref{lem:4.7}$ that for any $\bm{x} \in Q'_{1}$, 
if $D_{4}(\bm{x})$ has no prime factor of the form $8 k \pm 3$, 
then $D_{4}(\bm{x})$ has at least one prime factor of the form $p = a^{2} + b^{2} \equiv 1 \pmod{8}$ with $a + b \equiv \pm 3 \pmod{8}$. 
In the same way, we can obtain the same conclusion for any $\bm{x} \in Q'_{2}$. 
From the above, we have $m \in B$. 
\end{proof}

\begin{proof}[Proof of Lemma~$\ref{lem:4.2}$]
Suppose that $D_{4 \times 2 \times 2}(\bm{a}) = D_{4}(\bm{b}) D_{4}(\bm{c}) D_{4}(\bm{d}) D_{4}(\bm{e}) \in 2^{16} \mathbb{Z}_{{\rm odd}}$. 
From Corollary~$\ref{cor:2.1}$ and Lemma~$\ref{lem:2.5}$, 
we have $b_{0} + b_{2} \equiv b_{1} + b_{3} \pmod{2}$. 
Therefore, from Lemmas~$\ref{lem:4.8}$ and $\ref{lem:4.9}$, 
we have 
$$
D_{4 \times 2 \times 2}(\bm{a}) \in \left\{ 2^{16} (4 m + 1), \: 2^{16} (8 m + 3), \: 2^{16} m' \mid m \in \mathbb{Z}, m' \in A \cup B \right\}. 
$$
\end{proof}

To prove Lemma~$\ref{lem:4.3}$, 
we use the following lemma. 

\begin{lem}\label{lem:4.10}
Suppose that $\bm{x} = (x_{0}, x_{1}, x_{2}, x_{3}) \equiv (0, 0, 1, 1) \pmod{2}$, 
$(x_{0} + x_{2}) (x_{1} + x_{3}) \equiv \pm 1 \pmod{8}$ and $x_{0} \not\equiv x_{1} \pmod{4}$. 
Then $D_{4}(\bm{x})$ has at least one prime factor of the form $8 k - 3$. 
\end{lem}
\begin{proof}
From 
\begin{align*}
(x_{0} - x_{2}) (x_{1} - x_{3}) 
&= (x_{0} + x_{2}) (x_{1} + x_{3}) - 2 x_{0} x_{3} - 2 x_{2} x_{1} \\ 
&\equiv (x_{0} + x_{2}) (x_{1} + x_{3}) - 2 x_{0} x_{3} - 2 x_{2} (x_{0} + 2) \\ 
&\equiv (x_{0} + x_{2}) (x_{1} + x_{3}) - 2 x_{0} (x_{3} + x_{2}) - 4 x_{2} \\ 
&\equiv \pm 3 \pmod{8}, 
\end{align*}
we have $(x_{0} - x_{2})^{2} + (x_{1} - x_{3})^{2} \equiv - 6 \pmod{16}$. 
This completes the proof. 
\end{proof}

\begin{proof}[Proof of Lemma~$\ref{lem:4.3}$]
Suppose that $D_{4 \times 2 \times 2}(\bm{a}) = D_{4}(\bm{b}) D_{4}(\bm{c}) D_{4}(\bm{d}) D_{4}(\bm{e}) \in 2^{17} \mathbb{Z}_{{\rm odd}}$. 
Then, from Corollary~$\ref{cor:2.1}$ and Lemma~$\ref{lem:2.5}$, 
we have $b_{0} + b_{2} \equiv b_{1} + b_{3} \pmod{2}$. 
Therefore, from Remarks~$\ref{rem:2.2}$ and $\ref{rem:2.4}$~(1) and Lemma~$\ref{lem:4.4}$, 
we have $b_{0} + b_{2} \equiv b_{1} + b_{3} \equiv 1 \pmod{2}$. 
We may assume without loss of generality that $\bm{b} \equiv \bm{c} \equiv \bm{d} \equiv \bm{e} \equiv (0, 0, 1, 1) \pmod{2}$. 
Let 
\begin{align*}
Q_{3} &:= \left\{ (x_{0}, x_{1}, x_{2}, x_{3}) \in \mathbb{Z}^{4} \mid (x_{0}, x_{1}, x_{2}, x_{3}) \equiv (0, 0, 1, 1) \: \: ({\rm mod} \: 2), \right. \\ 
&\qquad \left. (x_{0} + x_{2}) (x_{1} + x_{3}) \equiv \pm 3 \: \: ({\rm mod} \: 8) \right\}, \\ 
Q_{4} &:= \left\{ (x_{0}, x_{1}, x_{2}, x_{3}) \in \mathbb{Z}^{4} \mid (x_{0}, x_{1}, x_{2}, x_{3}) \equiv (0, 0, 1, 1) \: \: ({\rm mod} \: 2), \right. \\ 
&\qquad \left. (x_{0} + x_{2}) (x_{1} + x_{3}) \equiv \pm 1 \: \: ({\rm mod} \; 8) \right\}, \\ 
Q'_{3} &:= \left\{ (x_{0}, x_{1}, x_{2}, x_{3}) \in Q_{3} \mid x_{0} \equiv x_{1} \: \: ({\rm mod} \: 4) \right\}, \\ 
Q'_{4} &:= \left\{ (x_{0}, x_{1}, x_{2}, x_{3}) \in Q_{4} \mid x_{0} \not\equiv x_{1} \: \: ({\rm mod} \: 4) \right\}. 
\end{align*}
From Lemma~$\ref{lem:4.4}$, 
three of $\bm{b}, \bm{c}, \bm{d}, \bm{e}$ are elements of $Q_{3}$ and the other one is an element of $Q_{4}$. 
Moreover, since $(b_{0} - b_{1}) + (c_{0} - c_{1}) + (d_{0} - d_{1}) + (e_{0} - e_{1}) \equiv 0 \pmod{4}$ from Remark~$\ref{rem:2.4}$~(2), 
we find that at least one of $\bm{b}, \bm{c}, \bm{d}, \bm{e}$ is an element of $Q'_{3} \cup Q'_{4}$. 
On the other hand, it follows from Lemmas~$\ref{lem:4.5}$ and $\ref{lem:4.10}$ that for any $\bm{x} \in Q'_{3} \cup Q'_{4}$, 
$D_{4}(\bm{x})$ has at least one prime factor of the form $8 k - 3$. 
From the above, there exist $p \in P$ and $m \in \mathbb{Z}$ satisfying $D_{4 \times 2 \times 2}(\bm{a}) = 2^{17} p (2 m + 1)$. 
\end{proof}

\section{Possible numbers}
In this section, 
we determine all possible numbers. 
Lemmas~$\ref{lem:3.1}$ and $\ref{lem:4.1}$--$\ref{lem:4.3}$ imply that $S \left( {\rm C}_{4} \times {\rm C}_{2}^{2} \right)$ does not include every integer that is not mentioned in the following Lemmas~$\ref{lem:5.1}$--$\ref{lem:5.3}$. 

\begin{lem}\label{lem:5.1}
For any $m \in \mathbb{Z}$, 
the following are elements of $S({\rm C}_{4} \times {\rm C}_{2}^{2})$: 
\begin{enumerate}
\item[$(1)$] $16 m + 1$; 
\item[$(2)$] $2^{16} (4 m + 1)$; 
\item[$(3)$] $2^{16} (4 m + 1) (8 n + 3)$; 
\item[$(4)$] $2^{18} (2 m + 1)$; 
\item[$(5)$] $2^{18} (2 m)$. 
\end{enumerate}
\end{lem}

\begin{lem}\label{lem:5.2}
For any $p \in P'$ and $m \in \mathbb{Z}$, 
we have $2^{16} p (4 m - 1) \in S({\rm C}_{4} \times {\rm C}_{2}^{2})$. 
\end{lem}

\begin{lem}\label{lem:5.3}
For any $p \in P$ and $m \in \mathbb{Z}$, 
we have $2^{17} p (2 m + 1) \in S({\rm C}_{4} \times {\rm C}_{2}^{2})$. 
\end{lem}

\begin{rem}\label{rem:5.4}
From Lemma~$\ref{lem:5.1}$~$(3)$, 
we have $2^{16} m \in S({\rm C}_{4} \times {\rm C}_{2}^{2})$ for any $m \in A$. 
Also, from Lemma~$\ref{lem:5.2}$, 
we have $2^{16} m \in S({\rm C}_{4} \times {\rm C}_{2}^{2})$ for any $m \in B$. 
\end{rem}

\begin{proof}[Proof of Lemma~$\ref{lem:5.1}$]
We obtain (1) from 
\begin{align*}
D_{4 \times 2 \times 2}(m + 1, m, \ldots, m) 
&= D_{4}(4 m + 1, 4 m, 4 m, 4 m) D_{4}(1, 0, 0, 0)^{3} \\ 
&= (8 m + 1)^{2} - (8 m)^{2} \\ 
&= 16 m + 1. 
\end{align*}
We obtain (2) from 
\begin{align*}
D_{4 \times 2 \times 2}(m + 1, m + 1, m + 2, m, \ldots, m) 
&= D_{4}(4 m + 1, 4 m + 1, 4 m + 2, 4 m) D_{4}(1, 1, 2, 0)^{3} \\ 
&= 2 \left\{ (8 m + 3)^{2} - (8 m + 1)^{2} \right\} \cdot (2^{4})^{3} \\ 
&= 2^{13} (32 m + 8) \\ 
&= 2^{16} (4 m + 1). 
\end{align*}
We obtain (3) from 
\begin{align*}
&D_{4 \times 2 \times 2}(\overbrace{m + n + 1, \ldots, m + n + 1}^{4}, \overbrace{m - n, \ldots, m - n}^{4}, \\ 
&\qquad m + n + 1, m + n, m + n + 1, m + n - 1, m - n - 1, m - n, m - n, m - n) \\ 
&= D_{4}(4 m + 1, 4 m + 1, 4 m + 2, 4 m) D_{4}(4 n + 3, 4 n + 1, 4 n + 2, 4 n) D_{4}(1, 1, 0, 2) D_{4}(- 1, 1, 0, 2) \\ 
&= 2 \left\{ (8 m + 3)^{2} - (8 m + 1)^{2} \right\} \cdot 2 \left\{ (8 n + 5)^{2} - (8 n + 1)^{2} \right\} \cdot (- 2^{4}) \cdot (- 2^{4}) \\ 
&= 2^{10} (32 m + 8) (64 n + 24) \\ 
&= 2^{16} (4 m + 1) (8 n + 3). 
\end{align*}
We obtain (4) from 
\begin{align*}
&D_{4 \times 2 \times 2}(m + 2, m, m + 2, m + 1, m, m, m, m, m + 1, m + 1, m, m + 1, m, m, m, m) \\ 
&\quad= D_{4}(4 m + 3, 4 m + 1, 4 m + 2, 4 m + 2) D_{4}(3, 1, 2, 2) D_{4}(1, - 1, 2, 0)^{2} \\ 
&\quad= 2 \left\{ (8 m + 5)^{2} - (8 m + 3)^{2} \right\} \cdot 2^{5} \cdot (2^{4})^{2} \\ 
&\quad= 2^{14} (32 m + 16) \\ 
&\quad= 2^{18} (2 m + 1). 
\end{align*}
We obtain (5) from 
\begin{align*}
&D_{4 \times 2 \times 2}(m + 1, m, m + 1, m, m, m - 1, m, m, m, m + 1, m, m + 1, m, m - 1, m - 1, m - 1) \\ 
&\quad= D_{4}(4 m + 1, 4 m - 1, 4 m, 4 m) D_{4}(1, 3, 2, 2) D_{4}(1, - 1, 2, 0) D_{4}(1, - 1, 0, - 2) \\ 
&\quad= 2 \left\{ (8 m + 1)^{2} - (8 m - 1)^{2} \right\} \cdot (- 2^{5}) \cdot 2^{4} \cdot (- 2^{4}) \\ 
&\quad= 2^{14} (32 m) \\ 
&\quad= 2^{18} (2 m). 
\end{align*}
\end{proof}

\begin{proof}[Proof of Lemma~$\ref{lem:5.2}$]
For any $p \in P'$, 
there exist $r, s \in \mathbb{Z}$ with $r \not\equiv s \pmod{2}$ satisfying $p = (4 r)^{2} + (4 s + 1)^{2}$. 
Let $k := \frac{r + s + 1}{2}$ and $l := \frac{r - s - 1}{2}$. 
Then we have 
\begin{align*}
2 p = (4 r + 4 s + 1)^{2} + (4 r - 4 s - 1)^{2} = (8 k - 3)^{2} + (8 l + 3)^{2}. 
\end{align*}
Therefore, from 
\begin{align*}
&D_{4 \times 2 \times 2}(k - m, l - m + 1, - k - m + 1, - l - m, k + m, l + m + 1, - k + m + 1, - l + m, \\ 
&\qquad k - m, l - m + 1, - k - m + 1, - l - m, k + m - 1, l + m, - k + m - 1, - l + m) \\
&\quad= D_{4}(4 k - 1, 4 l + 3, 2 - 4 k, - 4 l) D_{4}(1 - 4 m, 1 - 4 m, 2 - 4 m, - 4 m) \\ 
&\qquad \times D_{4}(1, 1, 2, 0) D_{4}(- 1, - 1, - 2, 0) \\ 
&\quad= - 2^{3} \left\{ (8 k - 3)^{2} + (8 l + 3)^{2} \right\} \cdot 2 \left\{ (3 - 8 m)^{2} - (1 - 8 m)^{2} \right\} \cdot 2^{4} \cdot 2^{4} \\ 
&\quad= - 2^{12} \left\{ (8 k - 3)^{2} + (8 l + 3)^{2} \right\} (- 32 m + 8) \\ 
&\quad= 2^{15} \left\{ (8 k - 3)^{2} + (8 l + 3)^{2} \right\} (4 m - 1), 
\end{align*}
we have $2^{16} p (4 m - 1) \in S({\rm C}_{4} \times {\rm C}_{2}^{2})$. 
\end{proof}

\begin{proof}[Proof of Lemma~$\ref{lem:5.3}$]
For any $p \in P$, 
there exist $r, s \in \mathbb{Z}$ with $r \equiv s \pmod{2}$ satisfying $p = (4 r + 2)^{2} + (4 s + 1)^{2}$. 
Let $k := \frac{r + s}{2}$ and $l := \frac{r - s}{2}$. 
Then we have 
\begin{align*}
2 p = (4 r + 4 s + 3)^{2} + (4 r - 4 s + 1)^{2} = (8 k + 3)^{2} + (8 l + 1)^{2}. 
\end{align*}
Therefore, from 
\begin{align*}
&D_{4 \times 2 \times 2}(m + l + 1, m + k + 1, m - l + 1, m - k, m + l + 1, m + k + 1, m - l + 1, m - k, \\ 
&\qquad m + l + 1, m + k + 1, m - l + 1, m - k, m + l, m + k, m - l - 1, m - k) \\ 
&\quad= D_{4}(4 m + 4 l + 3, 4 m + 4 k + 3, 4 m - 4 l + 2, 4 m - 4 k) D_{4}(1, 1, 2, 0)^{2} D_{4}(- 1, - 1, - 2, 0) \\ 
&\quad= \left\{ (8 m + 5)^{2} - (8 m + 3)^{2} \right\} \left\{ (8 l + 1)^{2} + (8 k + 3)^{2} \right\} \cdot (2^{4})^{2} \cdot 2^{4} \\ 
&\quad= 2^{12} (32 m + 16) \left\{ (8 k + 3)^{2} + (8 l + 1)^{2} \right\} \\ 
&\quad= 2^{16} \left\{ (8 k + 3)^{2} + (8 l + 1)^{2} \right\} (2 m + 1), 
\end{align*}
we have $2^{17} p (2 m + 1) \in S({\rm C}_{4} \times {\rm C}_{2}^{2})$. 
\end{proof}

From Lemmas~$\ref{lem:3.1}$, $\ref{lem:4.1}$--$\ref{lem:4.3}$ and $\ref{lem:5.1}$--$\ref{lem:5.3}$, 
Theorem~$\ref{thm:1.1}$ is proved.

\clearpage

\bibliography{reference}
\bibliographystyle{plain}

\medskip
\begin{flushleft}
Faculty of Education, 
University of Miyazaki, 
1-1 Gakuen Kibanadai-nishi, 
Miyazaki 889-2192, 
Japan \\ 
{\it Email address}, Yuka Yamaguchi: y-yamaguchi@cc.miyazaki-u.ac.jp \\ 
{\it Email address}, Naoya Yamaguchi: n-yamaguchi@cc.miyazaki-u.ac.jp 
\end{flushleft}

\end{document}